\documentclass[runningheads]{llncs}
\usepackage{graphicx}
\usepackage[utf8]{inputenc}
\usepackage{amsfonts}
\usepackage{amsmath}
\usepackage{todonotes}
\usepackage[ruled, linesnumbered]{algorithm2e}
\usepackage{hyperref}
\hypersetup{
    colorlinks=true,
    linkcolor=blue,
    filecolor=magenta,      
    urlcolor=cyan,
}
\usepackage{mathrsfs}

\newcommand{\RR}{\mathbb{R}}
\newcommand{\ZZ}{\mathbb{Z}}

\newcommand{\st}{{\mbox{s.t.}}}
\newcommand{\T}{{\mathsf{T}}}
\newcommand{\supp}{\operatorname{supp}}
\newcommand{\F}{\mathscr{F}}

\tikzstyle{integral}=[line width = 1.5 mm, dash pattern=on 6mm off 2pt, red]
\tikzstyle{half integral}=[line width = 1.5 mm, dash pattern=on 2mm off 3pt, blue]
\tikzstyle{two thirds}=[line width = 1.5 mm, dash pattern = on 6mm off 3pt on 2mm off 3pt, violet]
\tikzstyle{one third}=[line width = 1.5mm, dash pattern = on 2mm off 8pt, green]
\tikzstyle{integral thin}=[line width = 0.5 mm, dash pattern=on 6mm off 2pt, red]
\tikzstyle{half integral thin}=[line width = 0.5 mm, dash pattern=on 2mm off 3pt, blue]
\tikzstyle{three fifths}=[line width = 1.5 mm, dash pattern = on 1mm off 7pt on 4mm off 3pt, orange]
\tikzstyle{two fifths}=[line width = 1.5 mm, dash pattern = on 8mm off 6pt, pink]

\begin{document}
\title{Revisiting a Cutting Plane Method for Perfect Matchings\thanks{Supported by the Fields Institute for Research in Mathematical Sciences through the
2019 Fields Undergraduate Summer Research Program.}}
\author{Amber Q. Chen\inst{1}\orcidID{0000-0002-6032-6120} \and
Kevin K. H. Cheung\inst{2}\orcidID{0000-0002-4287-1766} \and
P. Michael Kielstra\inst{3}\orcidID{0000-0002-3854-7146} \and
Andrew Winn\inst{4}\orcidID{0000-0003-1242-8285}}
\authorrunning{A. Chen et al.}
\institute{University of Toronto, 40 St. George Street, Toronto, ON M5S 2E4, Canada 
\email{qingyuan.chen@mail.utoronto.ca} \and
Carleton University, 1125 Colonel By Dr, Ottawa, ON K1S 5B6, Canada
\email{kcheung@math.carleton.ca}\and
Harvard University, 1 Oxford Street, Cambridge, MA 02138, USA
\email{pmkielstra@college.harvard.edu}
\and
Texas A\&M University, Mailstop 3368, College Station, TX 77843–3368,  USA\\
\email{andrewwinn3@tamu.edu}}
\maketitle              % typeset the header of the contribution
\begin{abstract}
	In 2016, Chandrasekaran, Végh, and Vempala published a method to solve the 
	minimum-cost perfect matching problem on an arbitrary graph by solving a 
	strictly polynomial number of linear programs.  However, their method 
	requires a strong uniqueness condition, which they imposed by using 
	perturbations of the form $c(i)=c_0(i)+2^{-i}$.  On large graphs 
	(roughly $m>100$), these perturbations lead to cost values that exceed the 
	precision of floating-point formats used by typical linear programming 
	solvers for numerical calculations.  We demonstrate, by a sequence of 
	counterexamples, that perturbations are required for the algorithm to work, 
	motivating our formulation of a general method that arrives at the same 
	solution to the problem as Chandrasekaran \textit{et al}. but overcomes the 
	limitations described above by solving multiple linear programs without 
	using perturbations.  We then give an explicit algorithm that exploits are 
	method, and show that this new algorithm still runs in strongly polynomial 
	time.
\end{abstract}


\section{Introduction}

Given a graph $G=(V, E)$ with edge cost function $c$, the minimum-cost
(or minimum-weight) perfect matching problem is to find a perfect matching
$E' \subseteq E$ (a subset such that every vertex $v \in V$ is covered
by exactly one $uv \in E'$) so that the sum of the costs of $E'$ is minimized.
As mentioned in \cite{cook_rohe_1999}, the minimum-cost perfect matching 
problem is a classical problem in combinatorial optimization with numerous and 
varied applications.

Since Edmonds \cite{edmonds1965a} introduced the blossom algorithm (a 
polynomial-time combinatorial method of solving the problem), a number of 
efficient
implementations have been developed over the years, with Kolmogorov's Blossom 
V~\cite{kolmogorov_blossom_2009} being a recent notable version.

The problem can also be formulated as a binary integer program:
\begin{align*}
\min \sum_{e \in E}& c(e) x(e) \\
\st \sum_{uv \in E} x(uv) &= 1 & \forall~v \in V \\
x(e) &\in \{0,1\} &\forall~e \in E.
\end{align*}
To use linear programming (LP) techniques to solve the problem, the constraints
$x(e) \in \{0,1\}$ are first relaxed to $x(e) \in [0, 1]$ and then to
$x(e) \geq 0$ since the upper bounds are then implied.  The linear program 
that results turns out to be exact for bipartite graphs in the sense that a 
basic optimal solution is the incidence vector of a minimum-weight perfect 
matching.  Edmonds \cite{edmonds1965b} provides an LP formulation for 
non-bipartite graphs that has the same property.  It requires the addition of 
``blossom inequalities":
\begin{align*}
\min \sum_{e \in E}& c(e) x(e) \\
\st \sum_{uv \in E} x(uv) &= 1 &\forall~v \in V \\
\sum_{\substack{uv \in E \\ u \in S, v\notin S}} x(uv) &\geq 1, &\forall 
S\subseteq
V,~|S|\mbox{ odd},~ 3 \leq |S| \leq |V|-3 \\
x(e) &\geq 0 &\forall~e \in E.
\end{align*}
Unfortunately, the presence of an exponential number of constraints in
this formulation precludes polynomial-time solvability via a generic LP 
solver.  As a result, researchers in the past have  experimented with a 
cutting-plane approach, solving the relaxation first without the blossom 
inequalities, then iteratively finding and adding violated inequalities until the problem 
has an integral solution.  A polynomial-time (though impractical) algorithm 
follows using the equivalence of separation and optimization via the 
ellipsoid method (see Grötschel \textit{et al.}~\cite{GLS}) and the 
polynomial-time identification of violated blossom inequalities due to Padberg 
and Rao~\cite{padberg_rao_1982}.  The existence of a practical LP-based cutting 
plane method for the minimum-weight perfect matching remained uncertain until 
2016, when Chandrasekaran~\textit{et al.}~\cite{chandrasekaran_cutting_2016} 
gave a cutting-plane algorithm which uses only a polynomial number of linear 
programs.

Their approach involves carefully selecting the blossom inequalities to be
included at each iteration and requires that the optimal solution to the linear
program be unique. As this uniqueness property does not always hold in
general, their method introduces an edge ordering and a perturbation on the 
edge costs. (The edge costs are assumed to be integers.) In particular, if 
$c_0(i)$ is the original cost for the $i$-th edge, then the perturbed cost is 
$c(i)=c_0(i)+2^{-i}$.  Such a perturbation turns out to be sufficient for 
providing the required uniqueness property.  Even though the increase in 
size in representing the perturbed costs is polynomial, when the graph is large 
(say with hundreds of edges), the precision required to represent the 
perturbed costs exceeds what is typical of the floating-point formats used by 
most LP solvers \cite{gunluk_exact_2011}. (For example, $4+2^{-100}= 
\frac{5070602400912917605986812821505}{1267650600228229401496703205376}$ requires a mantissa of over 100 bits.)

To overcome the potential numerical difficulties caused by 
perturbation, we present a variant of the algorithm which does not require
an explicit perturbation to ensure uniqueness. It works instead by solving a
sequence of linear programs for each single linear program that the original
algorithm would solve.  We present a method whereby, given the solutions to 
these programs, we can derive the optimal solution to a hypothetical perturbed 
linear program without any explicit calculations on perturbed costs.  After 
this, the rest of the proof follows just as it did for the original algorithm.

The trade-off is that our algorithm has a worse runtime than that of 
Chandrasekaran \textit{et al.}  Theirs requires solving $O(n \log n)$ linear 
programs, while ours solves $O(mn \log n)$.  This is, however, still 
polynomial.  % Since the current best implementations of the Edmonds blossom 
% algorithm run faster than the Chandrasekaran algorithm anyway (see 
% \cite{kolmogorov_blossom_2009}), this is not a serious problem since we would 
% not use a linear-programming algorithm if asymptotic runtime were the major 
% concern.

The rest of this paper is organized as follows.  After defining some terms 
(Section~\ref{sec:prelim}) and 
summarizing the algorithm from \cite{chandrasekaran_cutting_2016} 
(Section~\ref{sec:cvv}), we give 
examples of graphs which show that this algorithm requires some form of perturbation in both the primal and dual problems.  
In particular, without perturbing the edge costs, we cannot guarantee that the 
intermediate solutions will always be half-integral (Section~\ref{NonHalf})
or that the algorithm will terminate (Section~\ref{sec:cycling}).  
This occurs even if we force the primal solution to be the same as it 
would have been with perturbations.  This motivates our new method, which uses 
multiple linear programs to accurately emulate the perturbations.  We first explain this in a general case 
(Section~\ref{sec:perturb}) and then apply it to the specific problem of finding
perfect matchings (Section~\ref{sec:newalg}).

\section{Notation and definitions}\label{sec:prelim}
The set of $m \times n$ matrices with real entries is denoted by $\RR^{m\times
n}$.

For a matrix $A\in\RR^{m\times n}$, $A_{i,j}$ denotes the $(i,j)$-entry
of $A$; that is, the entry of $A$ at the intersection of the $i$-th row and the
$j$-th column.  $A_{:,j}$ denotes the $j$-th column of $A$ and
$A_{i,:}$ the $i$-th row. The transpose of $A$ is denoted by $A^\T$.

Following common usage in combinatorics, for a finite set $E$, $\RR^E$ denotes
the set of tuples of real numbers indexed by elements of $E$.  For $y \in
\RR^E$, $y(i)$ denotes the entry indexed by $i \in E$.  For a positive integer
$n$, $\RR^n$ is an abbreviation for $\RR^{\{1,\ldots,n\}}.$ Depending on the
context, elements of $\RR^n$ are treated as if they were elements of
$\RR^{n\times 1}$.

We assume familiarity with basic terminology related to matchings and linear 
programming.  A refresher of the former can be found at \cite[Chapter 
5]{cook_combinatorial_1998}, and of the latter at 
\cite{schrijver_theory_2000}.  We next recall some definitions in 
Chandrasekaran \textit{et al.}~\cite{chandrasekaran_cutting_2016} to facilitate 
discussion of their minimum-cost perfect matching algorithm.

Let $G=(V,E)$ be an simple undirected graph with integer edge costs given by
$c \in \mathbb{Z}^E$. A family $\F$ of subsets of $V$ is said to be
\textit{laminar} if for all $U,W \in \F$, $U \cap W = \emptyset$ or $U
\subseteq W$ or $W \subseteq U$.  For a set $S \subseteq V$,
$\delta(S)$ denotes the set of edges incident to one vertex in $S$ and one
vertex not in $S$. For a vertex $u$, $\delta(u)$ denotes $\delta(\{u\})$.  For
$x \in \RR^E$ and $T \subseteq E$, $x(T)$ denotes the sum $\sum_{e \in T}
x(e)$.

Let $M$ be a matching of a graph $H = (V, E)$. Let $U \subseteq V$, and let $\F$ be a laminar family
of subsets of $V$. Then $M$ is a \textit{$(U, \F)$-perfect-matching} if
${|\delta(S) \cap M|} \leq 1$ for every $S \in \F$ and $M$ covers exactly the 
vertex set $U$.  A set of vertices $S \in \F$ is said to be \textit{$(H, 
\F)$-factor-critical} for a graph $H$ if, for every $u \in S$, there exists an 
$(S \setminus \{u\}, \F)$-perfect-matching using the edges of $H$.

For a laminar family $\F$ of odd subsets of $V$, define the 
following primal-dual pair of linear programming problems:

\begin{align*}
\min \sum_{uv\in E}& c(uv) x(uv)\tag{$P_\F(G, c)$}\label{Pf}\\
\st\ x(\delta(u))&=1&\forall u\in V\\
x(\delta(S))&\ge 1& \forall S\in {\F}\\
x&\ge0,\\
\\
\max \sum_{S\in V \cup \F}&\Pi(S)\tag{$D_\F(G, c)$}\label{Df}\\
\st\ \sum_{S\in V \cup \F:uv\in \delta(S)} \Pi(S)&\le c(uv) & \forall uv\in E \\
\Pi(S)&\ge0&\forall S\in \F.
\end{align*}

Let $\Pi$ be a feasible solution to \ref{Df}.
$G_\Pi$ denotes the graph $(V,E_\Pi)$ where
$E_\Pi = \{ uv \in E : 
    \sum_{S\in V \cup \F:uv\in \delta(S)} \Pi(S) = c(uv)\}$.  Colloquially,
$E_\Pi$ is the set of ``tight" edges with respect to $\Pi$.  We say that $\Pi$
is an \textit{$\F$-critical dual} if every $S \in \F$ is $(G_\Pi,
\F)$-factor-critical and $\Pi(T) > 0$ for every non-maximal $T \in \F$.  If
$\Pi$ is an $\F$-critical dual except that some sets $S \in \F$ for which
$\Pi(S)=0$ may not be $(G_\Pi, \F)$-factor-critical, we say that $\Pi$ is an
$\F$-positively-critical dual.

Finally, we define a metric on solutions to \ref{Df} 
$$\Delta(\Gamma, \Pi)=\sum_{S \in V \cup\mathscr{F}} \frac{1}{|S|} 
|\Gamma(S)-\Pi(S)|.$$  It can be easily verified that this has the properties of a metric.

For a given fixed $\Gamma$, we say that $\Pi$ is \textit{$\Gamma$-extremal} if 
it minimizes $\Delta(\Gamma, \Pi)$.  Given $\Gamma$ and a primal solution $x$, 
we may find a  $\Gamma$-extremal dual optimal solution by solving the following 
linear program \cite[Section 5]{chandrasekaran_cutting_2016}:

\begin{align*}
\min \sum_{S \in V \cup \mathscr{F}}&\frac{1}{|S|}r(S)\tag{$D^*_\F(G, 
c)$}\label{D*}\\
\st\ r(S)+\Pi(S)&\geq\Gamma(S)&\forall S \in V \cup \F_x\\
-r(S)+\Pi(S)&\leq\Gamma(S)&\forall S \in V \cup \F_x\\
\sum_{uv \in \delta(S)}\Pi(S)&=c(uv)&\forall uv \in \supp(x)\\
\sum_{uv \in \delta(S)}\Pi(S)&\leq c(uv)&\forall uv \notin \supp(x)\\
\Pi(S)&\geq0& \forall S \in \mathscr{F}_x\\
\Pi(S)&=0&\forall S \in \F \setminus \F_x\\
r(s)&=0&\forall S \in \F \setminus \F_x,
\end{align*}

where $\F_x=\{S \in \F : x(\delta(S))=1\}$.  The solution will give us values 
for $r$ and $\Pi$; we ignore $r$ and take $\Pi$ to be our $\Gamma$-extremal 
solution.

\section{The Chandrasekaran-Végh-Vempala algorithm}\label{sec:cvv}
Algorithm~\ref{alg:cvv} for finding a minimum-cost perfect matching on $G$
is due to Chandrasekaran \textit{et al.}~\cite{chandrasekaran_cutting_2016}.  
It assumes, as we will from now on, that the edge costs are integers.

\begin{algorithm}[H]\caption{C-P-Matching Algorithm}\label{alg:cvv}
	\KwIn{A graph $G=(V, E)$ with edge costs $c \in \ZZ^E$.}
	\KwOut{A binary vector $x$ representing a minimum-cost perfect matching on 
	$G$.}
	Let $c$ be the cost function on the edges after perturbation (i.e., after 
	ordering the edges arbitrarily and increasing the cost of each edge $i$ by 
	$2^{-i}$). \label{cvv:OrderEdges}
	
	$\F \leftarrow \emptyset$, $\Gamma \leftarrow 0$
	
	\While{$x$ is not integral}{
		Find an optimal solution $x$ to \ref{Pf}.\label{cvv:primalstep}
		
		Find a $\Gamma$-extremal dual optimal solution $\Pi$ to \ref{Df} 
		(possibly by solving \ref{D*}). 
		\label{cvv:extremal}
		
		$\mathscr{H}'\leftarrow\{S\in {\F}: \Pi(S)>0\}$ \label{cvv:H'}
		
		Let $\mathscr{C}$ denote the set of odd cycles in $\supp(x)$. For each 
		$C\in \mathscr{C}$, define $\hat C$ as the union of $V(C)$ and the	
		maximal sets of $\mathscr{H}'$ intersecting it.
		
		$\mathscr{H}''\leftarrow \{\hat C: C\in \mathscr{C}\}$
		
		$\mathscr{F} \leftarrow \mathscr{H}'\cup \mathscr{H}''$, $\Gamma 
		\leftarrow \Pi$
	}
	\KwRet{$x$}
\end{algorithm}

The authors of the algorithm showed that $\mathscr{F}$ is always a laminar 
family and that the algorithm terminates after $O(n \log n)$ iterations, 
assuming that \ref{Pf} has a unique optimal solution in every
iteration of the algorithm.  This is ensured through the use of perturbations
in the first step.  The authors further demonstrate that a $\Gamma$-extremal 
dual solution, with an $\F$-critical $\Gamma$, is an $\F$-positively-critical 
dual optimal to \ref{Df}, so the result of step~\ref{cvv:extremal} is 
$\F$-positively-critical.  When combined with the uniquness assumption, this 
leads to $x$ being half-integral in each iteration.

The choice of using powers of $\frac{1}{2}$ for the perturbations is to keep
the increases in input size polynomial. However, to guarantee uniqueness, 
powers of a sufficiently small $\epsilon > 0$ can be used instead.

\begin{lemma}\label{cvvepsilon}
There exists a $\delta>0$ such that the perturbations used in Algorithm 
\ref{alg:cvv} may be replaced with powers of $\epsilon$ for any 
$\delta>\epsilon>0$.
\end{lemma}
\begin{proof}
	Consider the proof given for the efficacy of the $2^{-i}$ perturbation in 
	\cite[Section 7]{chandrasekaran_cutting_2016}.  This uses only one property 
	of the perturbation: that, if $\sum_{i=1}^m a(i) 2^{-i}=\sum_{k=1}^n 
	b(k) 2^{-k}$, with $a(i), b(k) > 0$, then $m = n$ 
	and $a(i)=b(i)$ for all $i$.  We prove this for a class of arbitrary $\epsilon > 0$, after which the 
desired result follows.
	
	Assume $\sum_{i=1}^m a(i)\epsilon^i = \sum_{k=1}^n b(k) \epsilon^k$.  
	Assume 	further, without loss of generality, that $m \leq n$.  Then 
	$\sum_{i=1}^m (a(i)-b(i))\epsilon^i - \sum_{k=m+1}^n b(k) \epsilon^k = 0$.
	
	Take $m < n$.  For $\epsilon$ sufficiently small, either $a(i)=b(i)$ for 
	all $i \in \{1, \ldots, m\}$ or $\sum_{i=1}^m |(a(i)-b(i))|\epsilon^i >  
	\sum_{k=m+1}^n b(k) \epsilon^k$.  In the first case, $ \sum_{k=m+1}^n b(k) 
	\epsilon^k = 0$, a contradiction since $\epsilon$ and all $b(k)$ are 
	positive; in the second, $\sum_{i=1}^m (a(i)-b(i))\epsilon^i - 
	\sum_{k=m+1}^n b(k) \epsilon^k \neq 0$.  Therefore $m = n$.
	
	Assume there exists a minimal $l$ such that $a(l)-b(l) \neq 0$.  Then 
	\[
	0=\sum_{i=1}^m(a(i)-b(i))\epsilon^i = (a(l)-b(l))\epsilon^l + 
	\epsilon^{l+1}\sum_{i=l+1}^n(a(i)-b(i))\epsilon^{i-l-1}.
	\]
	For sufficiently small $\epsilon$, $|(a(l)-b(l))\epsilon^l| > 
	|\epsilon^{l+1}\sum_{i=l+1}^n(a(i)-b(i))\epsilon^{i-l-1}|$, so 
	$(a(l)-b(l))\epsilon^l + \epsilon^{l+1} \sum_{i=l+1}^n(a(i)-b(i)) 
	\epsilon^{i-l-1} \neq 0$.
	
	This shows that, for any given $a$ and $b$, there exists a $\delta$ such 
	that if $\sum a(i)\delta^i = \sum b(k)\delta^k,$ then $a=b$ for all 
	$\delta > \epsilon > 0$.  In fact, we need only consider the cases where 
	$a$ and $b$ are basic feasible solutions to \ref{Pf}, because if there 
	exists an optimal solution that is not a basic feasible solution then there 
	exist two distinct basic feasible solutions that are optimal.  Therefore, 
	if optimal basic feasible solutions are unique, so are optimal solutions in 
	general.
	
	Fix $\F$.  Then, because \ref{Pf} is bounded and finite-dimensional, it 
	has a finite number of basic feasible solutions $s_1, \dots, s_k$.  Every 
	pair $(s_p, s_q)$ gives us a $\delta$ by setting $a=s_p$, $b=s_q$ and 
	running through the logic above.  Take the smallest of these $\delta$s to 
	complete the proof.
\end{proof}

For reasons that will become clear later (see Section~\ref{sec:perturb}),
it will be more convenient to use powers of a sufficiently small $\epsilon > 0$
as perturbations instead of powers of $\frac{1}{2}$.  In any case, increasing 
the bit-length required to represent the edge costs can lead to practical 
computation challenges since most LP solvers employ fixed-length floating-point 
formats.  (Notable exceptions exist, such as 
QSopt-Exact~\cite{ApplegateDavidL.2007Estl} and the SoPlex rational 
solver~\cite{GleixnerAmbrosM.2016IRfL}, but they are significantly slower than 
non-exact solvers.)  We feel strongly that the key to a successful 
implementation of Algorithm~\ref{alg:cvv} is to not work with any explicit 
numerical perturbation.

An obvious way of modifying the algorithm is simply to not perturb the edge
costs and run the rest of the procedure as stated, but this violates the uniqueness assumption, and as easily demonstrated in \cite[Section 1]{chandrasekaran_cutting_2016}, can lead to non-half-integrality and cycling.

Instead, we may emulate perturbations by ordering the edges (as in step
\ref{cvv:OrderEdges} of Algorithm~\ref{alg:cvv}) and then finding a
lexicographically-minimal optimal solution to \ref{Pf}, where $c$ is now an
unperturbed cost function. This may be accomplished using Algorithm
\ref{alg:LexMinPrimal}, which shows the process in a more general case.

\begin{algorithm}[H]
	\caption{Lexicographically-Minimal Primal Algorithm}\label{alg:LexMinPrimal}
	\KwIn{A linear program $P$ of the form $\min c^\T x\ \st\ Ax \geq b$, where 
	$x \in \mathbb{R}^n$.}
	\KwOut{The lexicographically-minimal solution $x$ to $P$.}
	Solve $P$ and let its opimal value be $\gamma$.
	
	$K \leftarrow \emptyset$, $x \leftarrow 0$
	
	\For{$i \leftarrow 1$ \KwTo $n$}{
		Set $x$ to an optimal solution to
        \abovedisplayskip=0pt
		\belowdisplayskip=0pt
		\begin{align}
		\min \ & x_i\notag \\
		\st\ c^\T x &= \gamma \notag \\
		x_j &=z & \forall (j, z) \in K \notag\\
		Ax&\geq b. \notag
		\end{align}
		
		$K \leftarrow K \cup \{(i, x_i)\}$
	}

	\KwRet{$x$}
\end{algorithm}

By \cite[p. 138]{schrijver_theory_2000}, the 
lexicographically-minimal optimal solution to \ref{Pf} is the same 
as the optimal solution to the perturbed \ref{Pf}.  
Unfortunately, this on its own ensures neither half-integrality nor 
convergence: for an arbitrarily small nonzero value there exist graphs such 
that the lexicographically-minimal optimal solution to \ref{Pf} contains 
smaller values, and there exist graphs on which this modification of the 
algorithm enters an infinite loop.  Before giving a slightly more complex 
modification of the algorithm that uses a multi-stage approach to mimic solving 
with perturbation without actually working with perturbations, we first give 
some examples of graphs that demonstrate the problems just mentioned.

\section{Non-half-integral solution}\label{NonHalf}
The following example, which we call the ``dancing robot," shows that, if
the edge costs are not perturbed at all, having an $\F$-critical dual is not 
sufficient to guarantee that all lexicographically-minimal optimal primal 
solutions are half-integral. Chandrasekaran \textit{et 
al.}~\cite{chandrasekaran_cutting_2016} provide an example early in their paper 
of a graph on which their algorithm as written does not maintain 
half-integrality, but this does not entirely suffice for our purposes, as the 
lexicographically-minimal primal solution on this graph, for any edge ordering, 
is integral.

The graph shown in Figures \ref{DRPerfect}-\ref{DR3}, with all edges having 
cost 1, eventually gives non-half-integral values when run through the original 
algorithm without any perturbation while enforcing a lexicographically-minimal 
optimal primal. During each iteration, an optimal dual solution is given by 
$\Pi$, the vector having value $\frac{1}{2}$ on the entries indexed by the 
vertices and $0$ on entries indexed by the sets in $\F$.  Note that all edges 
in the graph are tight with respect to $\Pi$. We can see that, although the 
primal solutions in the first and second iterations are half-integral 
(shown in Figures~\ref{DR1} and~\ref{DR2}), the solution in the third iteration 
no longer is.  The $\frac{1}{3}$- and $\frac{2}{3}$-edges are shown in 
Figure~\ref{DR3}.

Meanwhile, the dual solution $\Pi$ is a positively-critical
optimal dual for the current $\F$ in every iteration, as well as a critical 
dual for the next $\F$. For instance, the $\Pi$ from the second iteration, 
feasible to the dual problems from both the second and third iterations, is 
trivially an $\F$-positively-critical optimal dual for the second iteration, 
since none of the sets $S \in \F$ have positive dual value. For the third 
iteration, since there exists an $(S\setminus\{u\}, \F)$-perfect-matching for 
any node $u \in S \in \F$, and since $\F$ only has maximal sets, that same 
$\Pi$ is an $\F$-critical dual.

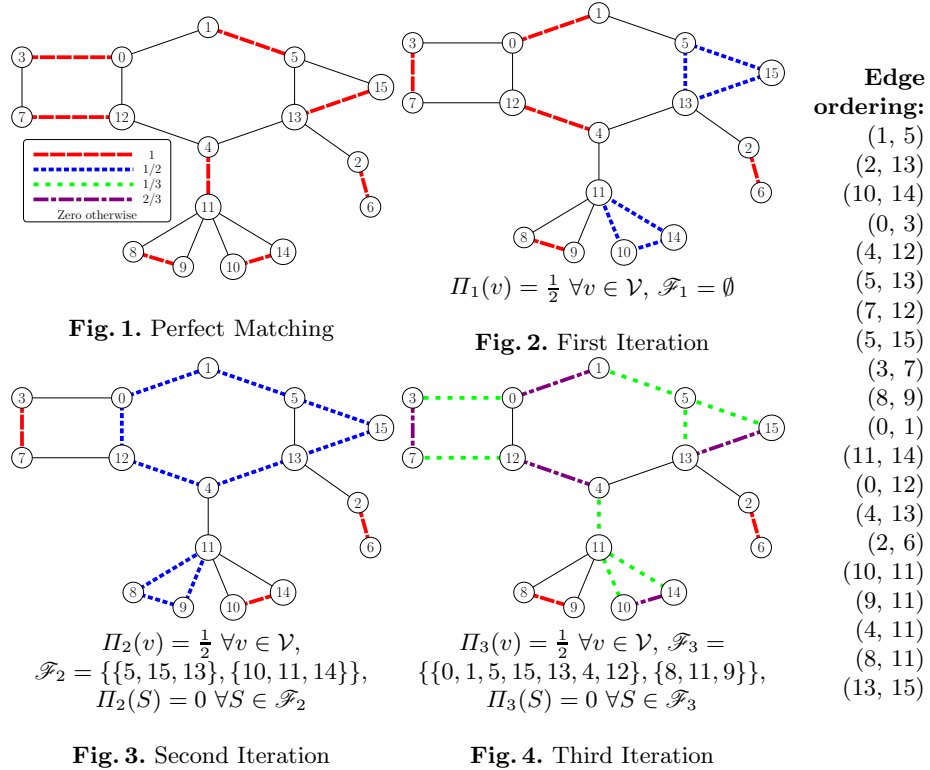
\begin{figure}
\begin{minipage}{0.84\textwidth}
\begin{minipage}{0.5\textwidth}
\begin{center}
\resizebox{\textwidth}{!}{
\begin{tikzpicture}[thick, yscale=0.6]
\tikzset{
  font={\fontsize{18pt}{12}\selectfont}}
\coordinate (0) at (-3.46, 2);
\coordinate (1) at (0, 4);
\coordinate (2) at (6, -5);
\coordinate (3) at (-7.46, 2);
\coordinate (4) at (0, -4);
\coordinate (5) at (3.46, 2);
\coordinate (6) at (6.5, -8);
\coordinate (7) at (-7.46, -2);
\coordinate (8) at (-3, -11);
\coordinate (9) at (-1, -12);
\coordinate (10) at (1, -12);
\coordinate (11) at (0, -8);
\coordinate (12) at (-3.46, -2);
\coordinate (13) at (3.46, -2);
\coordinate (14) at (3, -11);
\coordinate (15) at (6.92, 0);
 \draw (4) -- (12);
\draw[integral] (8) -- (9);
\draw (0) -- (12);
\draw[integral]  (0) -- (3);
\draw (0) -- (1);
\draw [integral] (1) -- (5);
\draw (2) -- (13);
\draw[integral] (2) -- (6);
\draw (3) -- (7);
\draw (4) -- (13);
\draw[integral] (4) -- (11);
\draw (5) -- (13);
\draw (5) -- (15);
\draw[integral] (7) -- (12);
\draw (8) -- (11);
\draw (9) -- (11);
\draw[integral] (10) -- (14);
\draw (10) -- (11);
\draw (11) -- (14);
\draw[integral] (13) -- (15);

\node at (0)[draw, thick, circle, fill=white, inner sep = 3.5pt]{0};
\node at (1)[draw, thick, circle, fill=white, inner sep = 3.5pt]{1};
\node at (2)[draw, thick, circle, fill=white, inner sep = 3.5pt]{2};
\node at (3)[draw, thick, circle, fill=white, inner sep = 3.5pt]{3};
\node at (4)[draw, thick, circle, fill=white, inner sep = 3.5pt]{4};
\node at (5)[draw, thick, circle, fill=white, inner sep = 3.5pt]{5};
\node at (6)[draw, thick, circle, fill=white, inner sep = 3.5pt]{6};
\node at (7)[draw, thick, circle, fill=white, inner sep = 3.5pt]{7};
\node at (8)[draw, thick, circle, fill=white, inner sep = 3.5pt]{8};
\node at (9)[draw, thick, circle, fill=white, inner sep = 3.5pt]{9};
\node at (10)[draw, thick, circle, fill=white, inner sep = 3.5pt]{10};
\node at (11)[draw, thick, circle, fill=white, inner sep = 3.5pt]{11};
\node at (12)[draw, thick, circle, fill=white, inner sep = 3.5pt]{12};
\node at (13)[draw, thick, circle, fill=white, inner sep = 3.5pt]{13};
\node at (14)[draw, thick, circle, fill=white, inner sep = 3.5pt]{14};
\node at (15)[draw, thick, circle, fill=white, inner sep = 3.5pt]{15};

% Legend
\draw [rounded corners = 4pt] (-7.4,-3.5) rectangle (-1.5, -9);
%\node at (-4.45, -4) {\large Legend of Primal Values};
\draw (-7, -4.5) -- (-3, -4.5) [integral];
\draw (-7, -5.5) -- (-3, -5.5) [half integral];
\draw (-7, -6.5) -- (-3, -6.5) [one third];
\draw (-7, -7.5) -- (-3, -7.5) [two thirds];
\node at (-2.25,-4.5){\Large 1};
\node at (-2.25,-5.5){\Large 1/2};
\node at (-2.25,-6.5){\Large 1/3};
\node at (-2.25,-7.5){\Large 2/3};
\node at (-4.45, -8.5) {\Large Zero otherwise};
\end{tikzpicture}}
\caption{Perfect Matching}
\label{Legend}
\label{DRPerfect}
\end{center}
\end{minipage}
\begin{minipage}{0.5\textwidth}
\begin{center}
\resizebox{\textwidth}{!}{
\begin{tikzpicture}[thick, yscale=0.6]
\tikzset{
  font={\fontsize{18pt}{12}\selectfont}}
\coordinate (0) at (-3.46, 2);
\coordinate (1) at (0, 4);
\coordinate (2) at (6, -5);
\coordinate (3) at (-7.46, 2);
\coordinate (4) at (0, -4);
\coordinate (5) at (3.46, 2);
\coordinate (6) at (6.5, -8);
\coordinate (7) at (-7.46, -2);
\coordinate (8) at (-3, -11);
\coordinate (9) at (-1, -12);
\coordinate (10) at (1, -12);
\coordinate (11) at (0, -8);
\coordinate (12) at (-3.46, -2);
\coordinate (13) at (3.46, -2);
\coordinate (14) at (3, -11);
\coordinate (15) at (6.92, 0);
\draw (4) -- (12)	[integral];
\draw (8) -- (9)	[integral];
\draw (0) -- (12);
\draw (0) -- (3);
\draw (0) -- (1)	[integral];
\draw (1) -- (5);
\draw (2) -- (13);
\draw (2) -- (6)	[integral];
\draw (3) -- (7)	[integral];
\draw (4) -- (13);
\draw (4) -- (11);
\draw (5) -- (13)	[half integral];
\draw (5) -- (15)	[half integral];
\draw (7) -- (12);
\draw (8) -- (11);
\draw (9) -- (11);
\draw (10) -- (14)	[half integral];
\draw (10) -- (11)	[half integral];
\draw (11) -- (14)	[half integral];
\draw (13) -- (15)	[half integral];

\node at (0)[draw, thick, circle, fill=white, inner sep = 3.5pt]{0};
\node at (1)[draw, thick, circle, fill=white, inner sep = 3.5pt]{1};
\node at (2)[draw, thick, circle, fill=white, inner sep = 3.5pt]{2};
\node at (3)[draw, thick, circle, fill=white, inner sep = 3.5pt]{3};
\node at (4)[draw, thick, circle, fill=white, inner sep = 3.5pt]{4};
\node at (5)[draw, thick, circle, fill=white, inner sep = 3.5pt]{5};
\node at (6)[draw, thick, circle, fill=white, inner sep = 3.5pt]{6};
\node at (7)[draw, thick, circle, fill=white, inner sep = 3.5pt]{7};
\node at (8)[draw, thick, circle, fill=white, inner sep = 3.5pt]{8};
\node at (9)[draw, thick, circle, fill=white, inner sep = 3.5pt]{9};
\node at (10)[draw, thick, circle, fill=white, inner sep = 3.5pt]{10};
\node at (11)[draw, thick, circle, fill=white, inner sep = 3.5pt]{11};
\node at (12)[draw, thick, circle, fill=white, inner sep = 3.5pt]{12};
\node at (13)[draw, thick, circle, fill=white, inner sep = 3.5pt]{13};
\node at (14)[draw, thick, circle, fill=white, inner sep = 3.5pt]{14};
\node at (15)[draw, thick, circle, fill=white, inner sep = 3.5pt]{15};
\end{tikzpicture}}
$\Pi_1(v)=\frac{1}{2}$  $\forall{v} \in \mathcal{V}$, $\F_1=\emptyset$
\caption{First Iteration}
\label{DR1}
\end{center}
\end{minipage}
\begin{minipage}{0.5\textwidth}
\begin{center}
\resizebox{\textwidth}{!}{
\begin{tikzpicture}[thick, yscale=0.6]
\tikzset{
  font={\fontsize{18pt}{12}\selectfont}}
\coordinate (0) at (-3.46, 2);
\coordinate (1) at (0, 4);
\coordinate (2) at (6, -5);
\coordinate (3) at (-7.46, 2);
\coordinate (4) at (0, -4);
\coordinate (5) at (3.46, 2);
\coordinate (6) at (6.5, -8);
\coordinate (7) at (-7.46, -2);
\coordinate (8) at (-3, -11);
\coordinate (9) at (-1, -12);
\coordinate (10) at (1, -12);
\coordinate (11) at (0, -8);
\coordinate (12) at (-3.46, -2);
\coordinate (13) at (3.46, -2);
\coordinate (14) at (3, -11);
\coordinate (15) at (6.92, 0);
\draw (4) -- (12)	[half integral];
\draw (8) -- (9)	[half integral];
\draw (0) -- (12)	[half integral];
\draw (0) -- (3);
\draw (0) -- (1) 	[half integral];
\draw (1) -- (5)	[half integral];
\draw (2) -- (13);
\draw (2) -- (6)	[integral];
\draw (3) -- (7) 	[integral];
\draw (4) -- (13)	[half integral];
\draw (4) -- (11);
\draw (5) -- (13);
\draw (5) -- (15)	[half integral];
\draw (7) -- (12);
\draw (8) -- (11)	[half integral];
\draw (9) -- (11)	[half integral];
\draw (10) -- (14)	[integral];
\draw (10) -- (11);
\draw (11) -- (14);
\draw (13) -- (15)	[half integral];

\node at (0)[draw, thick, circle, fill=white, inner sep = 3.5pt]{0};
\node at (1)[draw, thick, circle, fill=white, inner sep = 3.5pt]{1};
\node at (2)[draw, thick, circle, fill=white, inner sep = 3.5pt]{2};
\node at (3)[draw, thick, circle, fill=white, inner sep = 3.5pt]{3};
\node at (4)[draw, thick, circle, fill=white, inner sep = 3.5pt]{4};
\node at (5)[draw, thick, circle, fill=white, inner sep = 3.5pt]{5};
\node at (6)[draw, thick, circle, fill=white, inner sep = 3.5pt]{6};
\node at (7)[draw, thick, circle, fill=white, inner sep = 3.5pt]{7};
\node at (8)[draw, thick, circle, fill=white, inner sep = 3.5pt]{8};
\node at (9)[draw, thick, circle, fill=white, inner sep = 3.5pt]{9};
\node at (10)[draw, thick, circle, fill=white, inner sep = 3.5pt]{10};
\node at (11)[draw, thick, circle, fill=white, inner sep = 3.5pt]{11};
\node at (12)[draw, thick, circle, fill=white, inner sep = 3.5pt]{12};
\node at (13)[draw, thick, circle, fill=white, inner sep = 3.5pt]{13};
\node at (14)[draw, thick, circle, fill=white, inner sep = 3.5pt]{14};
\node at (15)[draw, thick, circle, fill=white, inner sep = 3.5pt]{15};
\end{tikzpicture}}
$\Pi_2(v)=\frac{1}{2}$  $\forall{v} \in \mathcal{V}$,
$\F_2 = \{\{5, 15, 13\}, \{10, 11, 14\}\}$, $\Pi_2(S)= 0 $  $\forall{S} \in \F_2$
\caption{Second Iteration}
\label{DR2}
\end{center}
\end{minipage}
\begin{minipage}{0.5\textwidth}
\begin{center}
\resizebox{\textwidth}{!}{
\begin{tikzpicture}[thick, yscale=0.6]
\tikzset{
  font={\fontsize{18pt}{12}\selectfont}}
\coordinate (0) at (-3.46, 2);
\coordinate (1) at (0, 4);
\coordinate (2) at (6, -5);
\coordinate (3) at (-7.46, 2);
\coordinate (4) at (0, -4);
\coordinate (5) at (3.46, 2);
\coordinate (6) at (6.5, -8);
\coordinate (7) at (-7.46, -2);
\coordinate (8) at (-3, -11);
\coordinate (9) at (-1, -12);
\coordinate (10) at (1, -12);
\coordinate (11) at (0, -8);
\coordinate (12) at (-3.46, -2);
\coordinate (13) at (3.46, -2);
\coordinate (14) at (3, -11);
\coordinate (15) at (6.92, 0);
\draw (4) -- (12)	[two thirds];
\draw (8) -- (9)	[integral];
\draw (0) -- (12);
\draw (0) -- (3)	[one third];
\draw (0) -- (1)	[two thirds];
\draw (1) -- (5)	[one third];
\draw (2) -- (13);
\draw (2) -- (6)	[integral];
\draw (3) -- (7)	[two thirds];
\draw (4) -- (13);
\draw (4) -- (11)	[one third];
\draw (5) -- (13)	[one third];
\draw (5) -- (15)	[one third];
\draw (7) -- (12)	[one third];
\draw (8) -- (11);
\draw (9) -- (11);
\draw (10) -- (14)	[two thirds];
\draw (10) -- (11)	[one third];
\draw (11) -- (14)	[one third];
\draw (13) -- (15)	[two thirds];

\node at (0)[draw, thick, circle, fill=white, inner sep = 3.5pt]{0};
\node at (1)[draw, thick, circle, fill=white, inner sep = 3.5pt]{1};
\node at (2)[draw, thick, circle, fill=white, inner sep = 3.5pt]{2};
\node at (3)[draw, thick, circle, fill=white, inner sep = 3.5pt]{3};
\node at (4)[draw, thick, circle, fill=white, inner sep = 3.5pt]{4};
\node at (5)[draw, thick, circle, fill=white, inner sep = 3.5pt]{5};
\node at (6)[draw, thick, circle, fill=white, inner sep = 3.5pt]{6};
\node at (7)[draw, thick, circle, fill=white, inner sep = 3.5pt]{7};
\node at (8)[draw, thick, circle, fill=white, inner sep = 3.5pt]{8};
\node at (9)[draw, thick, circle, fill=white, inner sep = 3.5pt]{9};
\node at (10)[draw, thick, circle, fill=white, inner sep = 3.5pt]{10};
\node at (11)[draw, thick, circle, fill=white, inner sep = 3.5pt]{11};
\node at (12)[draw, thick, circle, fill=white, inner sep = 3.5pt]{12};
\node at (13)[draw, thick, circle, fill=white, inner sep = 3.5pt]{13};
\node at (14)[draw, thick, circle, fill=white, inner sep = 3.5pt]{14};
\node at (15)[draw, thick, circle, fill=white, inner sep = 3.5pt]{15};
\end{tikzpicture}}
$\Pi_3(v)=\frac{1}{2}$  $\forall{v} \in \mathcal{V}$,
$\F_3 = \{\{0, 1, 5, 15, 13, 4, 12\}, \{8, 11, 9\}\}$, $\Pi_3(S)= 0 $  $\forall{S} \in \F_3$
\caption{Third Iteration}
\label{DR3}
\end{center}
\end{minipage}
\end{minipage}
\begin{minipage}[c]{0.15\textwidth}
\centering
\begin{flushright}
\textbf{Edge ordering:}\\
(1, 5)\\
(2, 13)\\
(10, 14)\\
(0, 3)\\
(4, 12)\\
(5, 13)\\
(7, 12)\\
(5, 15)\\
(3, 7)\\
(8, 9)\\
(0, 1)\\
(11, 14)\\
(0, 12)\\
(4, 13)\\
(2, 6)\\
(10, 11)\\
(9, 11)\\
(4, 11)\\
(8, 11)\\
(13, 15)\\
\end{flushright}
\end{minipage}
\end{figure}

Even worse, the algorithm will eventually enter into an infinite loop on this
example.  The details are tedious, so we will not go into them --- the example 
in the next section also loops but does not lose half-integrality.

On their own, then, a lexicographically-minimal primal and an $\F$-critical 
dual can guarantee neither half-integrality nor termination.  It is worth 
mentioning that we could expand the graph to get one with more 
non-half-integral edges by making the $2$-$6$ edge (the ``arm" of the dancing 
robot) overlap with another dancing robot's $6$-$2$ edge.  This combination 
would have twice as many non-half-integral edges, spread across twice as many 
non-half-integral paths, as the original dancing robot.  By combining multiple 
dancing robots in such a manner, we can get as many non-half-integral paths as 
we want, which indicates that we cannot avoid these non-half-integral edges via 
a simple combinatorial linear- or constant-time algorithm.

We can even alter the dancing robot in order to give us arbitrarily small but
nonzero values in a lexicographically-minimal optimal solution.  Say we want a 
primal solution $x$ such that, for some $uv$, $x(uv)=\frac{1}{2n+1}$ for some 
$n \in \ZZ_{\geq 0}$.  We add $2(n-1)$ new edges between $0$ and the $0$-$1$ 
edge, alternately without and with adjoining $4$-cycles. The next example shows 
how this works for $n=2$.

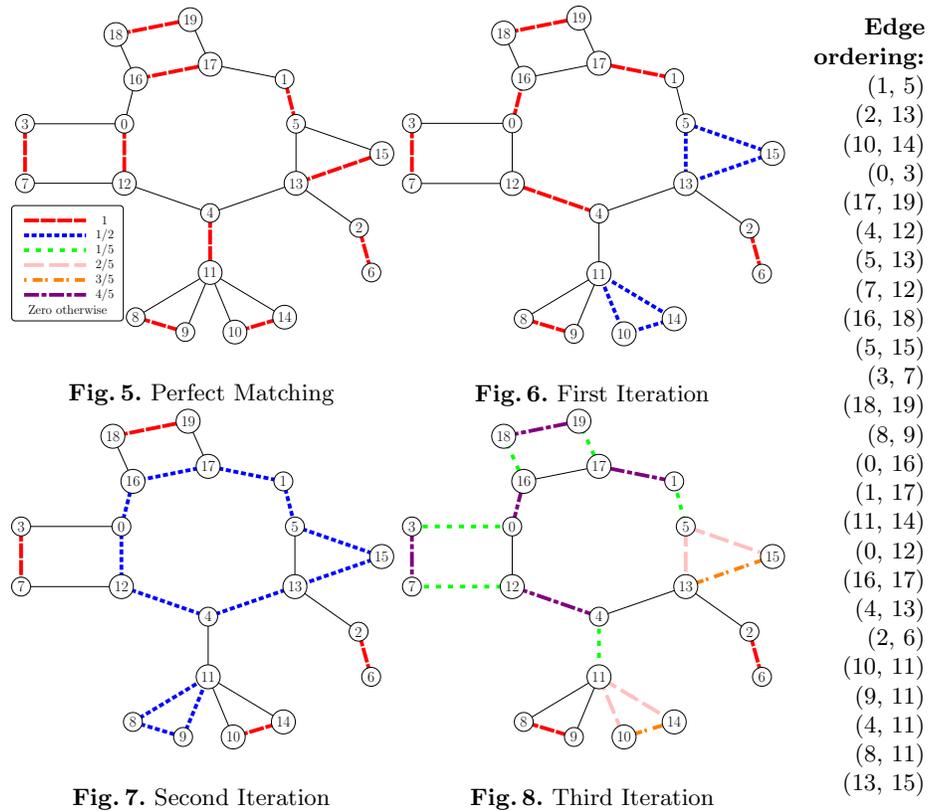
\begin{figure}
\begin{minipage}{0.84\textwidth}
\begin{minipage}{0.5\textwidth}
\begin{center}
\resizebox{\textwidth}{!}{
\begin{tikzpicture}[thick, yscale=0.6]
\tikzset{
  font={\fontsize{18pt}{12}\selectfont}}
\coordinate (0) at (-3.46, 2);
\coordinate (1) at (3, 5);
\coordinate (2) at (6, -5);
\coordinate (3) at (-7.46, 2);
\coordinate (4) at (0, -4);
\coordinate (5) at (3.46, 2);
\coordinate (6) at (6.5, -8);
\coordinate (7) at (-7.46, -2);
\coordinate (8) at (-3, -11);
\coordinate (9) at (-1, -12);
\coordinate (10) at (1, -12);
\coordinate (11) at (0, -8);
\coordinate (12) at (-3.46, -2);
\coordinate (13) at (3.46, -2);
\coordinate (14) at (3, -11);
\coordinate (15) at (6.92, 0);
\coordinate (16) at (-3, 5);
\coordinate (17) at (0, 6);
\coordinate (18) at (-3.8, 8);
\coordinate (19) at (-0.8, 9);

\draw (4) -- (12);
\draw (8) -- (9) [integral];
\draw (0) -- (12) [integral];
\draw (0) -- (3);
\draw (0) -- (16);
\draw (1) -- (5) [integral];
\draw (2) -- (13);
\draw (2) -- (6) [integral];
\draw (3) -- (7) [integral];
\draw (4) -- (13);
\draw (4) -- (11) [integral];
\draw (5) -- (13);
\draw (5) -- (15);
\draw (7) -- (12);
\draw (8) -- (11);
\draw (9) -- (11);
\draw (10) -- (14) [integral];
\draw (10) -- (11);
\draw (11) -- (14);
\draw (13) -- (15) [integral];
\draw (16) -- (17) [integral];
\draw (1) -- (17);
\draw (16) -- (18);
\draw (18) -- (19) [integral];
\draw (17) -- (19);

\foreach \n in {0,...,19}
	\node at (\n)[draw, thick, circle, fill=white, inner sep = 2.8pt]{\n};
% Legend
\draw [rounded corners = 4pt] (-8,-3.5) rectangle (-3.5, -11.25);
\draw (-7.5, -4.5) -- (-5, -4.5) [integral];
\draw (-7.5, -5.5) -- (-5, -5.5) [half integral];
\draw (-7.5, -6.5) -- (-5, -6.5) [one third];
\draw (-7.5, -7.5) -- (-5, -7.5) [two fifths];
\draw (-7.5, -8.5) -- (-5, -8.5) [three fifths];
\draw (-7.5, -9.5) -- (-5, -9.5) [two thirds];
\node at (-4.25,-4.5){\Large 1};
\node at (-4.25,-5.5){\Large 1/2};
\node at (-4.25,-6.5){\Large 1/5};
\node at (-4.25,-7.5){\Large 2/5};
\node at (-4.25,-8.5){\Large 3/5};
\node at (-4.25,-9.5){\Large 4/5};
\node at (-5.75,-10.5) {\Large Zero otherwise};
\end{tikzpicture}}
\caption{Perfect Matching}
\end{center}
\end{minipage}
\begin{minipage}{0.5\textwidth}
\begin{center}
\resizebox{\textwidth}{!}{
\begin{tikzpicture}[thick, yscale=0.6]
\tikzset{
  font={\fontsize{18pt}{12}\selectfont}}
\coordinate (0) at (-3.46, 2);
\coordinate (1) at (3, 5);
\coordinate (2) at (6, -5);
\coordinate (3) at (-7.46, 2);
\coordinate (4) at (0, -4);
\coordinate (5) at (3.46, 2);
\coordinate (6) at (6.5, -8);
\coordinate (7) at (-7.46, -2);
\coordinate (8) at (-3, -11);
\coordinate (9) at (-1, -12);
\coordinate (10) at (1, -12);
\coordinate (11) at (0, -8);
\coordinate (12) at (-3.46, -2);
\coordinate (13) at (3.46, -2);
\coordinate (14) at (3, -11);
\coordinate (15) at (6.92, 0);
\coordinate (16) at (-3, 5);
\coordinate (17) at (0, 6);
\coordinate (18) at (-3.8, 8);
\coordinate (19) at (-0.8, 9);

\draw (4) -- (12) [integral];
\draw (8) -- (9) [integral];
\draw (0) -- (12);
\draw (0) -- (3);
\draw (0) -- (16) [integral];
\draw (1) -- (5) ;
\draw (2) -- (13);
\draw (2) -- (6) [integral];
\draw (3) -- (7) [integral];
\draw (4) -- (13);
\draw (4) -- (11);
\draw (5) -- (13) [half integral];
\draw (5) -- (15) [half integral];
\draw (7) -- (12);
\draw (8) -- (11);
\draw (9) -- (11);
\draw (10) -- (14) [half integral];
\draw (10) -- (11) [half integral];
\draw (11) -- (14) [half integral];
\draw (13) -- (15) [half integral];
\draw (16) -- (17);
\draw (1) -- (17) [integral];
\draw (16) -- (18);
\draw (18) -- (19) [integral];
\draw (17) -- (19);

\foreach \n in {0,...,19}
	\node at (\n)[draw, thick, circle, fill=white, inner sep = 2.8pt]{\n};
\end{tikzpicture}}
\caption{First Iteration}
\end{center}
\end{minipage}
\begin{minipage}{0.5\textwidth}
\begin{center}
\resizebox{\textwidth}{!}{
\begin{tikzpicture}[thick, yscale=0.6]
\tikzset{
  font={\fontsize{18pt}{12}\selectfont}}
\coordinate (0) at (-3.46, 2);
\coordinate (1) at (3, 5);
\coordinate (2) at (6, -5);
\coordinate (3) at (-7.46, 2);
\coordinate (4) at (0, -4);
\coordinate (5) at (3.46, 2);
\coordinate (6) at (6.5, -8);
\coordinate (7) at (-7.46, -2);
\coordinate (8) at (-3, -11);
\coordinate (9) at (-1, -12);
\coordinate (10) at (1, -12);
\coordinate (11) at (0, -8);
\coordinate (12) at (-3.46, -2);
\coordinate (13) at (3.46, -2);
\coordinate (14) at (3, -11);
\coordinate (15) at (6.92, 0);
\coordinate (16) at (-3, 5);
\coordinate (17) at (0, 6);
\coordinate (18) at (-3.8, 8);
\coordinate (19) at (-0.8, 9);

\draw (4) -- (12) [half integral];
\draw (8) -- (9) [half integral];
\draw (0) -- (12) [half integral];
\draw (0) -- (3);
\draw (0) -- (16) [half integral];
\draw (1) -- (5)  [half integral];
\draw (2) -- (13);
\draw (2) -- (6) [integral];
\draw (3) -- (7) [integral];
\draw (4) -- (13) [half integral];
\draw (4) -- (11);
\draw (5) -- (13);
\draw (5) -- (15) [half integral];
\draw (7) -- (12);
\draw (8) -- (11) [half integral];
\draw (9) -- (11) [half integral];
\draw (10) -- (14) [integral];
\draw (10) -- (11);
\draw (11) -- (14);
\draw (13) -- (15) [half integral];
\draw (16) -- (17) [half integral];
\draw (1) -- (17) [half integral];
\draw (16) -- (18);
\draw (18) -- (19) [integral];
\draw (17) -- (19);

\foreach \n in {0,...,19}
	\node at (\n)[draw, thick, circle, fill=white, inner sep = 2.8pt]{\n};
\end{tikzpicture}}
\caption{Second Iteration}
\label{alteredrobot}
\end{center}
\end{minipage}
\begin{minipage}{0.5\textwidth}
\begin{center}
\resizebox{\textwidth}{!}{
\begin{tikzpicture}[thick, yscale=0.6]
\tikzset{
  font={\fontsize{18pt}{12}\selectfont}}
\coordinate (0) at (-3.46, 2);
\coordinate (1) at (3, 5);
\coordinate (2) at (6, -5);
\coordinate (3) at (-7.46, 2);
\coordinate (4) at (0, -4);
\coordinate (5) at (3.46, 2);
\coordinate (6) at (6.5, -8);
\coordinate (7) at (-7.46, -2);
\coordinate (8) at (-3, -11);
\coordinate (9) at (-1, -12);
\coordinate (10) at (1, -12);
\coordinate (11) at (0, -8);
\coordinate (12) at (-3.46, -2);
\coordinate (13) at (3.46, -2);
\coordinate (14) at (3, -11);
\coordinate (15) at (6.92, 0);
\coordinate (16) at (-3, 5);
\coordinate (17) at (0, 6);
\coordinate (18) at (-3.8, 8);
\coordinate (19) at (-0.8, 9);

\draw (4) -- (12) [two thirds];
\draw (8) -- (9) [integral];
\draw (0) -- (12);
\draw (0) -- (3) [one third];
\draw (0) -- (16) [two thirds];
\draw (1) -- (5)  [one third];
\draw (2) -- (13);
\draw (2) -- (6) [integral];
\draw (3) -- (7) [two thirds];
\draw (4) -- (13);
\draw (4) -- (11) [one third];
\draw (5) -- (13) [two fifths];
\draw (5) -- (15) [two fifths];
\draw (7) -- (12) [one third];
\draw (8) -- (11);
\draw (9) -- (11);
\draw (10) -- (14) [three fifths];
\draw (10) -- (11) [two fifths];
\draw (11) -- (14) [two fifths];
\draw (13) -- (15) [three fifths];
\draw (16) -- (17);
\draw (1) -- (17) [two thirds];
\draw (16) -- (18) [one third];
\draw (18) -- (19) [two thirds];
\draw (17) -- (19) [one third];

\foreach \n in {0,...,19}
	\node at (\n)[draw, thick, circle, fill=white, inner sep = 2.8pt]{\n};
\end{tikzpicture}}
\caption{Third Iteration}
\end{center}
\end{minipage}
\end{minipage}
\begin{minipage}[c]{0.15\textwidth}
\centering
\begin{flushright}
\textbf{Edge ordering:}\\
(1, 5)\\
(2, 13)\\
(10, 14)\\
(0, 3)\\
(17, 19)\\
(4, 12)\\
(5, 13)\\
(7, 12)\\
(16, 18)\\
(5, 15)\\
(3, 7)\\
(18, 19)\\
(8, 9)\\
(0, 16)\\
(1, 17)\\
(11, 14)\\
(0, 12)\\
(16, 17)\\
(4, 13)\\
(2, 6)\\
(10, 11)\\
(9, 11)\\
(4, 11)\\
(8, 11)\\
(13, 15)
\end{flushright}
\end{minipage}
\end{figure}

Simply by following the algorithm through, we see that we will eventually end 
up with a cut ($\{4, 12, 0, 16, 17, 1, 5, 15, 
13\}$ in Figure~\ref{alteredrobot}) with $2n+1$ edges coming out of it.  
Furthermore, by the conditions of the matching ($x(\delta(u))=1$) and the fact 
that these edges form a path in the matching, each edge coming out of this cut 
must have the same value, which we will call $\zeta$.  Since $x(\delta(S)) \geq 
1$, the minimum cost is when $(2n+1)\zeta=1$ or $\zeta=\frac{1}{2n+1}$.

\section{Cycling example}\label{sec:cycling}

Even when seeking a lexicographically-minimal optimal solution to \ref{Pf} 
with half-integrality maintained throughout, cycling can still occur in the 
absence of perturbation. The graph in Figures \ref{Step1} and \ref{Step2} 
(which is easily seen to have a perfect matching), with each edge having cost 
$1$, exhibits such behavoir.  At all times, an optimal $\F$-positively critical dual is given by a 
vector with the vertices having value $\frac{1}{2}$ and the odd sets in $\F$ 
having value $0$. Since Algorithm \ref{alg:cvv} only retains cuts which have 
nonzero values in the dual (step \ref{cvv:H'}), no cuts are preserved between 
iterations.  Thus, the blossom inequalities which were violated in the previous 
iteration are once again allowed to be violated in the next iteration, leading 
to cycling.

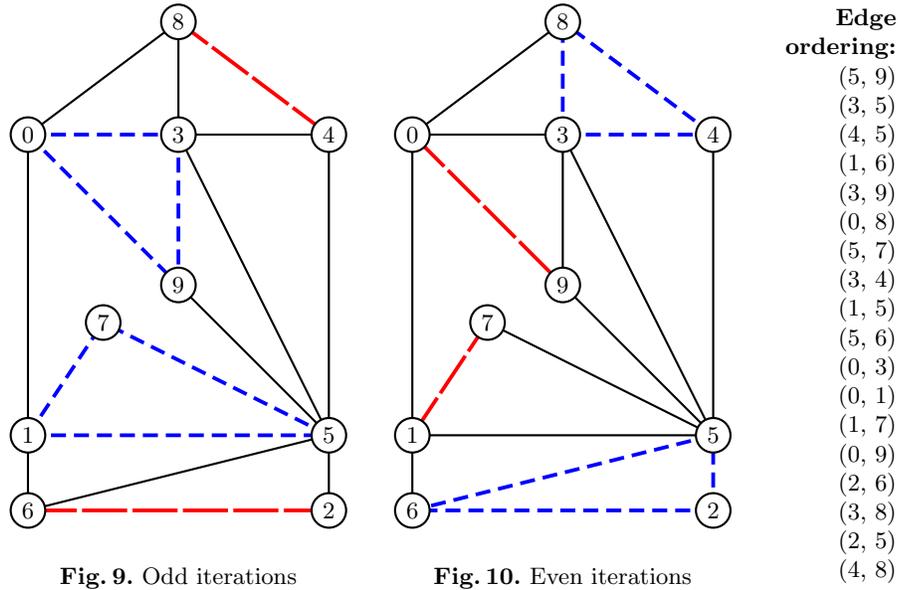
\begin{figure}
\begin{center}
\begin{minipage}{0.41\textwidth}
\centering
\begin{tikzpicture}[thick, xscale=1]
\coordinate (Zero) at (1,4);
\coordinate (One) at (1,0);
\coordinate (Two) at (5,-1);
\coordinate (Three) at (3,4);
\coordinate (Four) at (5,4);
\coordinate (Five) at (5,0);
\coordinate (Six) at (1,-1);
\coordinate (Seven) at (2,1.5);
\coordinate (Eight) at (3,5.5);
\coordinate (Nine) at (3,2);
\draw (One) -- (Five)	[half integral thin];
\draw (Five) -- (Six);
\draw (Three) -- (Five);
\draw (Four) -- (Eight)	[integral thin];
\draw (Zero) -- (Nine)	[half integral thin];
\draw (Two) -- (Five);
\draw (Five) -- (Seven)	[half integral thin];
\draw (Zero) -- (Three)	[half integral thin];
\draw (One) -- (Seven)	[half integral thin];
\draw (Zero) -- (Eight);
\draw (One) -- (Six);
\draw (Five) -- (Nine);
\draw (Three) -- (Four);
\draw (Three) -- (Eight);
\draw (Four) -- (Five);
\draw (Two) -- (Six)	[integral thin];
\draw (Three) -- (Nine)	[half integral thin];
\draw (Zero) -- (One);
\node at (Zero)[draw, circle, fill=white, inner sep=2pt]{0};
\node at (One)[draw, circle, fill=white, inner sep=2pt]{1};
\node at (Two)[draw, circle, fill=white, inner sep=2pt]{2};
\node at (Three)[draw, circle, fill=white, inner sep=2pt]{3};
\node at (Four)[draw, circle, fill=white, inner sep=2pt]{4};
\node at (Five)[draw, circle, fill=white, inner sep=2pt]{5};
\node at (Six)[draw, circle, fill=white, inner sep=2pt]{6};
\node at (Seven)[draw, circle, fill=white, inner sep=2pt]{7};
\node at (Eight)[draw, circle, fill=white, inner sep=2pt]{8};
\node at (Nine)[draw, circle, fill=white, inner sep=2pt]{9};
\end{tikzpicture}
\caption{Odd iterations}
\label{Step1}
\end{minipage}
\begin{minipage}{0.41\textwidth}
\centering
\begin{tikzpicture}[thick, xscale=1]
\coordinate (Zero) at (1,4);
\coordinate (One) at (1,0);
\coordinate (Two) at (5,-1);
\coordinate (Three) at (3,4);
\coordinate (Four) at (5,4);
\coordinate (Five) at (5,0);
\coordinate (Six) at (1,-1);
\coordinate (Seven) at (2,1.5);
\coordinate (Eight) at (3,5.5);
\coordinate (Nine) at (3,2);
\draw (One) -- (Five);
\draw (Five) -- (Six)	[half integral thin];
\draw (Three) -- (Five);
\draw (Four) -- (Eight)	[half integral thin];
\draw (Zero) -- (Nine)	[integral thin];
\draw (Two) -- (Five)	[half integral thin];
\draw (Five) -- (Seven);
\draw (Zero) -- (Three);
\draw (One) -- (Seven)	[integral thin];
\draw (Zero) -- (Eight);
\draw (One) -- (Six);
\draw (Five) -- (Nine);
\draw (Three) -- (Four)	[half integral thin];
\draw (Three) -- (Eight)[half integral thin];
\draw (Four) -- (Five);
\draw (Two) -- (Six)	[half integral thin];
\draw (Three) -- (Nine);
\draw (Zero) -- (One);
\node at (Zero)[draw, circle, fill=white, inner sep=2pt]{0};
\node at (One)[draw, circle, fill=white, inner sep=2pt]{1};
\node at (Two)[draw, circle, fill=white, inner sep=2pt]{2};
\node at (Three)[draw, circle, fill=white, inner sep=2pt]{3};
\node at (Four)[draw, circle, fill=white, inner sep=2pt]{4};
\node at (Five)[draw, circle, fill=white, inner sep=2pt]{5};
\node at (Six)[draw, circle, fill=white, inner sep=2pt]{6};
\node at (Seven)[draw, circle, fill=white, inner sep=2pt]{7};
\node at (Eight)[draw, circle, fill=white, inner sep=2pt]{8};
\node at (Nine)[draw, circle, fill=white, inner sep=2pt]{9};
\end{tikzpicture}
\caption{Even iterations}
\label{Step2}
\end{minipage}
\begin{minipage}[c]{0.15\textwidth}
\raggedleft
\textbf{Edge ordering:}\\
(5, 9)\\
(3, 5)\\
(4, 5)\\
(1, 6)\\
(3, 9)\\
(0, 8)\\
(5, 7)\\
(3, 4)\\
(1, 5)\\
(5, 6)\\
(0, 3)\\
(0, 1)\\
(1, 7)\\
(0, 9)\\
(2, 6)\\
(3, 8)\\
(2, 5)\\
(4, 8)
\end{minipage}
\end{center}
\end{figure}
\sloppy
This precludes us from ensuring termination of an approach using a 
lexicographically-minimal primal with an unperturbed dual.  More significantly, 
it also means we could not even implement a heuristic version which, in the 
event of cycling, would restart the algorithm with randomized edge orderings. 
Were half-integrality to occur in tandem with cycling, as it does in Section 
\ref{NonHalf}, we could simply verify half-integrality at each iteration, and, 
in rare cases of non-half-integrality, begin the entire algorithm again with a 
different edge ordering, giving us a good average-case runtime.\footnote{We 
discovered the dancing robot after searching through over two thousand randomly-generated graphs, all of which were rapidly and correctly solved by the use of 
a lexicographically-minimal primal and unperturbed dual. Furthermore, if a random edge ordering is applied to the dancing robot for example, it is very likely that a perfect matching will be found without issue.}  However, this graph 
shows that when a possibility of cycling exists, it is likely undetectable by 
any means other than direct comparison between iterations. This forces us to 
adopt an algorithm which simulates perturbations in the dual.

\section{Solving LP problems with perturbed costs}\label{sec:perturb}

Chandrasekaran \textit{et al.}~\cite{chandrasekaran_cutting_2016} chose a
specific perturbation of the costs, namely, adding $2^{-i}$ on each edge $i$.  
In general, perturbation in linear programming (usually for the purpose of
eliminating degeneracy, as in \cite{charnes_optimality_1952}) is of the form 
$\epsilon^i$ where $\epsilon$ is sufficiently small.  In theoretical analysis, 
$\epsilon$ is simply left unspecified. In the same spirit, we show in this
section how we could obtain optimal solutions to both the primal and dual
problems with perturbed costs, working with the fact that $\epsilon$ is
sufficiently small yet not given exactly, without the need of an optimal basis.
The method we describe is therefore able to avoid working 
directly with cost values that exceed the representation capacity of 
fixed-length floating-point formats typically used by LP solvers.  Our method 
is applicable to any situation in which the objective function of a generic LP 
problem is perturbed in order to enforce uniqueness of the optimal solution.  
In the next section, we specialize it to Algorithm \ref{alg:cvv}.

Let $A \in \RR^{m\times n}$, 
$b \in \RR^m$, and
$c_0,\ldots,c_k \in \RR^n$ for some nonnegative integer $k$.
Let $N \subseteq \{1,\ldots,n\}$.
Let $F = \{1,\ldots,n\}\setminus N$.
Define $c_\epsilon$ as $\sum_{p = 0}^k c_p \epsilon^p$ where $\epsilon \geq 0$.

Consider the linear programming problem:
\begin{align*}
\min \ c_\epsilon^\T x \tag{$P(\epsilon)$}\label{eqn:Peps}\\
\st \
A x & \geq  b \\
 x_j & \geq 0 & \forall~j \in N.
\end{align*}
Its dual is
\begin{align*}
\max \ & y^\T b \tag{$D(\epsilon)$}\label{eqn:Deps}\\
\st \  y^\T A_{:,j} & \leq c_\epsilon(j) & \forall~j \in N \\
y^\T A_{:,j} & = c_\epsilon(j) & \forall~j \in F \\
y &\geq 0.
\end{align*}

\begin{algorithm}[H]\label{alg:perturb}
    \SetAlgoLined
    \caption{Algorithm for perturbed LP primal-dual pair}
    \KwIn{\ref{eqn:Peps} with $\epsilon > 0$ sufficiently small.}
    \KwOut{An optimal $x'$ to \ref{eqn:Peps} and an optimal $y'$ to 
    \ref{eqn:Deps}.}
    
    $E \leftarrow \emptyset$, $J \leftarrow \emptyset$

    \For{$p \leftarrow 0$ \KwTo $k$}{
        $\overline{J} \leftarrow N \setminus J$
        
        \abovedisplayskip=0pt
		\belowdisplayskip=-\baselineskip
        Set $x_p$ to an optimal solution to
        \begin{align*}
           \min  \displaystyle\sum_{j \in \overline{J}} & c_p(j) x(j) \\
           \st 
            \displaystyle\sum_{j \in \overline{J}}{A_{i,j}}  x(j) & \geq b(i) & \forall~i \notin E \\
            \displaystyle\sum_{j \in \overline{J}}{A_{i,j}} x(j)& = b(i) & \forall~i \in E \\
            x(j)& \geq 0 & \forall~j \in \overline{J}
        \end{align*}
        \label{alg:perturb:primal}\\
        and $y_p$ to an optimal solution to its dual.

        $E \leftarrow E\cup \{ i : y_{p}(i) > 0\}$

        $J \leftarrow J \cup \{ j : {y_{p}}^\T A_{:,j} < b(j)\}$
    }

    Form $x' \in \RR^n$ such that $x'(j) = x_k(j)$ for all $j \notin J$ and 
    $x'(j) = 0$ for all $j \in J$.

    $y' \leftarrow \displaystyle\sum_{p=0}^k \epsilon^p y_p$
    
    \KwRet{$x', y'$}
\end{algorithm}

The correctness of Algorithm \ref{alg:perturb} follows from 
Lemma~\ref{cheunglp} below.  Before we give the proof, we illustrate the 
algorithm with an example.  For each $p$,  let $M_p$ denote the LP problem in 
step \ref{alg:perturb:primal} of the algorithm. Consider \ref{eqn:Peps} with
\begin{align*}
A = \begin{pmatrix} 1 & 0 & 1 \\ 0 & 1 & 2 \end{pmatrix},&&
b = \begin{pmatrix} 1 \\ 1 \end{pmatrix},&&
c_0 = \begin{pmatrix} 1 \\ 1 \\ 3 \end{pmatrix},&&
c_1 = \begin{pmatrix} 4 \\ 2 \\ 0 \end{pmatrix},&&
c_2 = \begin{pmatrix} -2 \\ -1 \\ 1 \end{pmatrix},&&
N = \{1,2\}.
\end{align*}
  The dual problem is
\[\begin{array}{rrrrcl}
\max & y(1) & + & y(2) \\
\st 
&   y(1) &   &      & \leq & 1+4\epsilon-2\epsilon^2 \\
&        &   & y(2) & \leq & 1+2\epsilon-\epsilon^2 \\
&   y(1) & + & 2y(2)& = &3+\epsilon^2 \\
& y(1) &, & y(2) & \geq & 0.
\end{array}\]
Note that $x_0 = \begin{pmatrix} 1 & 1 & 0\end{pmatrix}^\T$
and $y_0= \begin{pmatrix} 1 & 1 \end{pmatrix}^\T$
are optimal solutions to $P(0)$ and $D(0)$, respectively, which are in turn equivalent to $M_0$ and its dual.

Since $y_0(1), y_0(2) > 0$ and all the constraints in
$D(0)$ are satisfied with equality at $y_0$, $M_1$ is
\[\begin{array}{rrrrrrcl}
\min & 4x(1) & + & 2x(2)  \\
\st
&  x(1) &  &      & + & x(3) & = & 1 \\
&       &  & x(2) & + & 2x(3) & = & 1 \\
& x(1)  &, & x(2) &   &       & \geq & 0.
\end{array}\]
The dual of $M_1$ is
\[\begin{array}{rrrrcl}
\max & y(1) & + & y(2) \\
\st 
&   y(1) &   &      & \leq & 4 \\
&        &   & y(2) & \leq & 2 \\
&   y(1) & + & 2y(2)& = & 0.
\end{array}\]
An optimal solution to $M_1$ is $x_1 = \begin{pmatrix} \frac{1}{2} & 0 &
\frac{1}{2}\end{pmatrix}^\T$.  An optimal dual solution is
$y_1 =  \begin{pmatrix} 4 & -2\end{pmatrix}^\T$.
The second constraint in the dual is not active at $y_1$.
Hence, $M_2$ is
\[\begin{array}{rrrrcl}
\min & -2x(1) & + & x(3)  \\
\st
&  x(1) & + & x(3) & = & 1 \\
&       &   & 2x(3) & = & 1 \\
& x(1)  &   &       & \geq & 0.
\end{array}\]
The dual of $M_2$ is
\[\begin{array}{rrrrcl}
\max & y(1) & + & y(2) \\
\st 
&   y(1) &   &      & \leq & -2 \\
&   y(1) & + & 2y(2)& = & 1.
\end{array}\]
An optimal solution to $M_2$ is
$\begin{pmatrix} \frac{1}{2} & 0 & \frac{1}{2}\end{pmatrix}^\T$.
An optimal dual solution is
$y_2 =  \begin{pmatrix} -2 & \frac{3}{2} \end{pmatrix}^\T$.

Setting 
\[
y' = y_0 + \epsilon y_1 + \epsilon^2 y_2 =\begin{pmatrix}
1+ 4\epsilon - 2\epsilon^2 \\
1 -2\epsilon + \frac{3}{2}\epsilon^2
\end{pmatrix},\]
we have that $y'$ is a feasible solution to \ref{eqn:Deps} 
and satisfies complementary slackness with
$x' = \begin{pmatrix} \frac{1}{2} & 0 & \frac{1}{2}\end{pmatrix}^\T$ for the
primal-dual pair \ref{eqn:Peps} and \ref{eqn:Deps}
for a sufficiently small $\epsilon > 0$.

\begin{lemma}\label{cheunglp}
Let $M_p$ denote the LP problem solved in step \ref{alg:perturb:primal}
of Algorithm~\ref{alg:perturb}.
\begin{enumerate}
\item
For every $p \in \{1,\ldots,k\}$,
$x_p$ is an optimal solution to $M_0,\ldots,M_{p-1}$.
\item $x'$ and $y'$ are feasible to
    \ref{eqn:Peps} and \ref{eqn:Deps}, respectively,
     and satisfy complementary slackness.
\end{enumerate}
\end{lemma}

\begin{proof}

For each $j = 1,\ldots,p$, $M_j$ is obtained from $M_{j-1}$
by adding constraints to enforce complementary slackness with $y_{j-1}$.  
Removing $x(j)$ can be viewed as adding the constraint $x(j) = 0$.  It
follows that $x_p$ is feasible to $M_j$ and satisfies complementary slackness
with $y_j$ for all $j \in \{0,\ldots,p-1\}$.

To prove the second part, we start by noting that $x'$ is feasible to \ref{eqn:Peps}.  Let $E_p$, $J_p$, and $\overline{J_p}$ be the sets $E$, $J$, and $\overline{J}$ referred to in $M_p$.  The dual of $M_p$ is 
\begin{align*}
 \max \ & y^\T b \\
 \st \
   y^\T A_{:,j}& \leq c_p(j) & \forall~j \in \overline{J_p} \\
   y^\T A_{:,j}& = c_p(j) & \forall~j \in F \\
   y(i) &\geq 0 & \forall~i \notin E_p.
\end{align*}

Clearly, $y'^\T A_{:,j} = c_\epsilon(j)$ for all $j \in F$.
Next, we show that $y'^\T A_{:,j} \leq c_\epsilon(j)$ for all $j \in N$.

Suppose that $j \in \overline{J_k}$.
Since $\overline{J_p} \subseteq \overline{J_{p-1}}$ for $p = 1,\ldots, k$,
we have $y_p^\T A_{:,j} \leq c_p(j)$.
Thus, $$y'^\T A_{:,j} =
\sum_{p = 0}^k \epsilon^p y_p^\T A_{:,j} 
\leq \displaystyle\sum_{p = 0}^k \epsilon^p c_p(j) = c_\epsilon(j).$$

Now, suppose that $j \in J_k$. Then, there exists $r < n$ such that
$y_r^\T A_{:,j} < c_r(j)$.
Let $s_i = c_i(j) - y_i^\T A_{:,j}$ for $i = 1,\ldots,m$.
Thus, 
\begin{align*}
c_\epsilon(j) - {y'}^\T A_{:,j}
& = \sum_{p = 0}^k \epsilon^p s_p \\
& = \sum_{p = 0}^r \epsilon^p s_p + \sum_{p = r+1}^k \epsilon^p s_p \\
& \geq \epsilon^r \left(s_r + \sum_{p = r+1}^k \epsilon^{p-r} s_p\right) \\
& = \epsilon^r \left(s_r + \epsilon\sum_{q = 0}^{k-r-1} \epsilon^{q} s_{q+r+1}\right) \\
& > 0
\end{align*}
for $\epsilon > 0$ sufficiently small.

We now show that $y' \geq 0$.
Consider $y'(j)$ for some $j \in \{1,\ldots,m\}$.  If $j \notin E_k$,
then $y_p(j) \geq 0$ for $p = 0,\ldots,k$, implying that $y'(j) \geq 0$.
Otherwise, $j \in E_r$ for some $r \in \{1,\ldots,k\}$. Choose $r$
as small as possible.  We must have $y_{r-1}(j) > 0$.
Then, 
\begin{align*}
y'(j)
& = \sum_{p = 0}^k \epsilon^p y_p(j) \\
& \geq \sum_{p = r}^k \epsilon^p y_p(j) \\
& = \epsilon^r \left(y_r(j)_ + \epsilon\sum_{p = 0}^{k-r-1} \epsilon^{p} y_{p+r+1}(j)\right) \\
& > 0
\end{align*}
for $\epsilon > 0$ sufficiently small.

Finally, to see that $x'$ and $y'$ satisfy complementary slackness, note that,
by part 1, if $x'(j) > 0$, then $y_p^\T A_{:,j} = c_p(j)$ for all $p \in 
\{0,\ldots,k\}$.  Thus ${y'}^\T A_{:,j} = c_\epsilon(j)$.  Furthermore, if 
$A_{i,:}x' < b(i)$ for some $i$, ${y_p}(i) = 0$ for all $p \in 
\{0,\ldots,k\}$.  This implies that $y'(i) = 0$.
\end{proof}

We now make two observations that will be useful in the next section.  First, 
we can see from the proof of Lemma~\ref{cheunglp} that the dual of $M_p$ can be 
obtained directly from the dual of $M_{p-1}$ and an associated optimal solution 
$y_{p-1}$ by removing constraints (including the nonnegativity bound 
constraints) that are not active at $y_{p-1}$.  It follows that one can work 
exclusively with the duals of $M_0,\ldots,M_k$ if one is only interested in 
obtaining an optimal solution to \ref{eqn:Deps}.  Moreover, in practice, $y'$ 
hardly needs to be worked out for a particular value of $\epsilon$ and can be 
represented by the list $y_0,\ldots,y_k$.  Then, to determine if $y'(i) \neq 0$ 
for some $i$, simply check if there exists a $p$ such that $y_p(i) \neq 0$, 
since, for sufficiently small $\epsilon$, $y'(i) = 0$ if and only if $y_p(i) = 
0$ for all $p$.

\section{Modified Chandrasekaran-Végh-Vempala algorithm}\label{sec:newalg}

We now modify Algorithm~\ref{alg:cvv} to circumvent the need to utilize an 
explicit perturbation of the edge costs.  First, we arbitrarily order the edges 
and increase the cost of each edge $i$ by $\epsilon^i$ for some sufficiently 
small $\epsilon > 0$ that will remain unspecified.  By Lemma~\ref{cvvepsilon}, 
we may assume that with such a perturbation, Algorithm~\ref{alg:cvv} will still 
return a minimum-cost perfect matching.  In Algorithm~\ref{alg:cvv}, 
step~\ref{cvv:primalstep} and step~\ref{cvv:extremal} involve solving \ref{Pf} 
and \ref{D*} respectively with perturbed data. We emulate perturbations in the 
first of these by finding a lexicographically-minimal optimal solution.  The 
other is handled through the method developed in the previous section applied 
to the following LP, which is easily seen to be equivalent to \ref{D*}:
\begin{align*}
	\max \displaystyle \sum_{S \in \mathscr{V} \cup \F_x} & -\frac{1}{|S|}r(S) \\
    \st \ -r(S)-\Pi(S) & \leq-\Gamma(S) & \forall~S \in \mathscr{V} \cup \F_x \\
         -r(S)+\Pi(S) & \leq\Gamma(S) & \forall~S \in \mathscr{V} \cup \F_x \\
          \displaystyle \sum_{S \in \mathscr{V} \cup \F_x: uv \in \delta(S)}\Pi(S) & =c(uv) & \forall~uv \in \supp(x) \\
          \displaystyle \sum_{S \in \mathscr{V} \cup \F_x: uv \in \delta(S)}\Pi(S) & \leq c(uv) & \forall~uv \in E\setminus\supp(x) \\
            \Pi(S) & \geq0 & \forall~S \in \F_x, \\
           r(S) & \geq0 & \forall~S \in \mathscr{V} \cup \F_x
\end{align*}
where $\mathscr{V} = \{ \{v\} : v \in V \}$ and $\F_x = \{ S \in \F :
x(\delta(S)) = 1\}$.  With explicit perturbation of the edge costs, $\Gamma$ 
and $c$ will be polynomials in $\epsilon$.  Intuitively, we define $\Gamma_i$ 
and $c_i$ to be the coefficients of $\epsilon^i$ in $\Gamma$ and $c$; we will 
define these rigorously in a moment.

The reason for writing \ref{D*} as above is to make it plain that it can be 
viewed as the dual problem \ref{eqn:Deps} of some 
\ref{eqn:Peps} with cost values given by polynomials in $\epsilon$.
However, in an actual algorithm as seen below, we can work directly with
\ref{D*} as originally written.
With these changes and the following definitions, we obtain
Algorithm~\ref{alg:new}.

Given an ordering $\sigma : E \mapsto \{1, \dots, |E|\}$ on the edges of $G$, 
define the following cost function:
$$ c_i(uv)=\begin{cases}
c(uv) & i = 0\\
1 & i > 0,\ \sigma(uv) = i\\
0 & i > 0,\ \sigma(uv) \neq i
\end{cases}$$

With this, we define the following linear program:
\begin{align*}
	\min \displaystyle \sum_{S \in \mathscr{V} \cup \F_x} & \frac{1}{|S|}r(S)
	\tag{$D^i_\mathscr{F}(G, c, \sigma, \Gamma, L, M, N, Q)$}\label{Di} \\
    \st \ r(S)+\Pi(S) & \geq\Gamma_i(S) & \forall~S \in (\mathscr{V} \cup \F_x)\setminus L \\
         -r(S)+\Pi(S) & \leq\Gamma_i(S) & \forall~S \in (\mathscr{V} \cup \F_x)\setminus M \\
          \displaystyle \sum_{S \in \mathscr{V} \cup \F_x: uv \in 
          \delta(S)}\Pi(S) & =c_i(uv) & \forall~uv \in \supp(x) \\
          \displaystyle \sum_{S \in \mathscr{V} \cup \F_x: uv \in 
          \delta(S)}\Pi(S) & \leq c_i(uv) & \forall~uv \notin \supp(x) \cup N \\
            \Pi(S) & \geq0 & \forall~S \in \F_x\setminus Q. 
\end{align*}

Intuitively, $c_i$ and $\Gamma_i$ correspond to the coefficients of 
$\epsilon^i$ in $c$ and $\Gamma$ if we were to perturb the edge costs on the 
graph by $\epsilon^i$ and run Algorithm \ref{alg:cvv}.

\begin{algorithm}[H]
	\caption{Unperturbed C-P-Matching Algorithm}\label{alg:new}
	\KwIn{A graph $G=(V, E)$ with edge costs $c \in \ZZ^E$ and an 
	ordering $\sigma : E \mapsto \{1, \dots, |E|\}$.}
	\KwOut{A binary vector $x$ representing a minimum-cost perfect matching on $G$.}
	$\F \leftarrow \emptyset$; $\Gamma_0, \dots, \Gamma_{|E|} \leftarrow 0$
	
	\While{$x$ is not integral}{
		Let $x$ be the lexicographically-minimal optimal solution to \ref{Pf} 
		with respect to $\sigma$.
		
		$L \leftarrow \emptyset; M \leftarrow \emptyset; N \leftarrow 
		\emptyset; Q \leftarrow \emptyset; D_0, \dots, D_{|E|} \leftarrow 0$
		
		$\F_x \leftarrow \{ S \in \F : x(\delta(S)) = 1\}$
		
		\For{$i \leftarrow 0$ \KwTo $|E|$}{
			\label{dualline}Obtain an optimal solution $r, \Pi$ to \ref{Di}.\\
			\vspace{\baselineskip}
			$L \leftarrow L \cup \{S \in V \cup \F_x : 
			r(S)+\Pi(S)\neq\Gamma_{i}(S)\}$ \label{removeconstraints_start}
			
			$M \leftarrow M \cup \{S \in V \cup \F_x : -r(S) + \Pi(S) \neq 
			\Gamma_i(S)\}$
		
			$N \leftarrow N \cup \{uv \in E : \sum_{uv \in \delta(S)} \Pi(S) \neq c_i(uv)\}$
			
			$Q \leftarrow Q \cup \{S \in \F_x : \Pi(S) \neq 0 
			\}$\label{removeconstraints_end}\\
			\vspace{\baselineskip}
			$D_i \leftarrow \Pi$
		}
		$\mathscr{H}' \leftarrow \{S \in \F : \exists\ i\ \mathrm{s.t.}\ D_i(S) > 0\}$
		
		Let $\mathscr{C}$ be the set of odd cycles in $\supp(x)$.  For each $C \in \mathscr{C}$, let $V(C)$ be the union of $C$ with all sets in $\mathscr{H}'$ intersecting it.
		
		$\mathscr{H}'' \leftarrow \{V(C) : C \in \mathscr{C}\}$
		
		$\F \leftarrow \mathscr{H}' \cup \mathscr{H}''$
		
		$\Gamma \leftarrow D$
	}
\end{algorithm}

Steps \ref{removeconstraints_start} through \ref{removeconstraints_end} exist 
to remove the slack constraints from the next iterations of \ref{Di}, as in 
Algorithm \ref{alg:perturb}.

A reference implementation, written in Python 3, is available at 
\cite{kielstra_code_2019}.

\begin{lemma}
	In every iteration of the Unperturbed C-P-Matching Algorithm (\ref{alg:new}), $x$ is equal to its 
	counterpart in the C-P-Matching Algorithm (\ref{alg:cvv}) with perturbations $c(i)=\epsilon^i$.
\end{lemma}
\begin{proof}
As mentioned in Section \ref{sec:cvv}, by \cite{schrijver_theory_2000}, the 
lexicographically-minimal unperturbed primal solution is equal to the 
unique perturbed optimal primal solution for a given $\F$, so we need only 
show that $\F$ is always equal to its counterpart.
	
Consider $\Gamma$ as a single vector of polynomials in $\epsilon$, with
the coefficients of the $\epsilon^i$ terms given by $\Gamma_i$.  Then, by Lemma
\ref{cheunglp}, $y=\sum_i \epsilon^iD_i$ is an optimal solution to the linear
program in step \ref{dualline}, and $y(S)>0$ if and only if $\max (D_i) > 0$.
But the linear program in question is exactly that which the C-P-Matching
algorithm uses to obtain a $\Gamma$-extremal dual optimal solution.  Therefore
$\mathscr{H}'$, which is defined solely based on whether $y(S)>0$ or not, is
equal to its counterpart in the C-P-Matching algorithm.  Since $\mathscr{H}''$
is defined exactly the same way as its counterpart, the two are equal, so $\F$
is equal to its counterpart as well.
\end{proof}

Since, by Lemma \ref{cvvepsilon}, neither the correctness nor the complexity of 
Algorithm \ref{alg:cvv} are affected by changing from the perturbation 
$c(i)=2^{-i}$ to the perturbation $c(i)=\epsilon^i$, we can rephrase this to 
give

\begin{theorem}
The Unperturbed C-P-Matching Algorithm gives a minimum-cost perfect matching.
\end{theorem}

The lemma also has the following

\begin{corollary}
\sloppy
	The Unperturbed C-P-Matching algorithm requires solving $O(mn \log n)$ linear programming problems in the worst case.
\end{corollary}
\begin{proof}
According to \cite[Theorem 1]{chandrasekaran_cutting_2016}, the
C-P-Matching Algorithm takes at most $O(n \log n)$ iterations.  The Unperturbed
C-P-Matching Algorithm has the same number of iterations, but each iteration
utilizes $2(m+1)$ linear programming problems, which is $O(m)$.  Therefore, the
Unperturbed C-P-Matching Algorithm requires solving $O(m) \times O(n \log n) =
O(mn \log n)$ linear programming problems in total.
\end{proof}

\section{Final remarks}
We have developed a general method for solving perturbed linear programs in 
polynomial time without performing explicitly perturbed calculations, and 
demonstrated that it applies to the minimum-cost perfect matching problem.  The 
use of perturbations to guarantee uniqueness is common in linear programming, 
and it remains to be seen whether or not our method could be applied to other 
algorithms or used to solve other problems.  We do not yet know if our new 
algorithm, when properly implemented and optimized, can be made competitive 
with combinatorial methods such as Edmonds's blossom algorithm.  According to 
\cite{kolmogorov_blossom_2009}, the best known asymptotic runtime for such an 
implementation is $O(n(m + \log n))$.  Our algorithm, which solves $O(mn \log 
n)$ linear programs, each of which requires the use of a theoretically 
polynomial-time solver, is significantly slower in the worst case.

We encountered a number of interesting phenomena regarding the subroutine for
finding the lexicographically-minimal primal optimal solution.  Although, as
written, 
it requires solving a fixed number of linear programs ($|E|+1$), we noticed in 
empirical testing that it often gave this solution far more quickly than that, 
with the last few linear programs all giving the same answer.  We did not 
investigate this any further, but hypothesize that shortcuts exist
% that a better understanding of why it happened could potentially be used 
to decrease the runtime by a factor 
of $\frac{1}{4}$ or more.

\section*{Acknowledgements}
The authors would like to thank James Addis, as well as the organizers and 
supporting organizations of the Fields Undergraduate Summer Research Program 
--- Brittany Camp, Bryan Eelhart, Huaxiong Huang, Michael McCulloch, and the 
Fields Institute and Fields Centre for Quantitative Analysis and Modelling 
(Fields CQAM) --- without whom this research would not have been possible.

\bibliographystyle{splncs04}
\bibliography{bibliography}

\begin{thebibliography}{10}
\providecommand{\url}[1]{\texttt{#1}}
\providecommand{\urlprefix}{URL }
\providecommand{\doi}[1]{https://doi.org/#1}

\bibitem{ApplegateDavidL.2007Estl}
Applegate, D.L., Cook, W., Dash, S., Espinoza, D.G.: Exact solutions to linear
  programming problems. Operations Research Letters  \textbf{35}(6),  693--699
  (2007)

\bibitem{chandrasekaran_cutting_2016}
Chandrasekaran, K., Végh, L.A., Vempala, S.S.: The cutting plane method is
  polynomial for perfect matchings. Mathematics of Operations Research
  \textbf{41}(1),  23--48 (Feb 2016). \doi{10.1287/moor.2015.0714},
  \url{http://pubsonline.informs.org/doi/10.1287/moor.2015.0714}

\bibitem{charnes_optimality_1952}
Charnes, A.: Optimality and degeneracy in linear programming. Econometrica
  \textbf{20}(2),  160--170 (Apr 1952). \doi{10.2307/1907845}

\bibitem{cook_rohe_1999}
Cook, W., André, R.: Computing minimum-weight perfect matchings. INFORMS
  Journal on Computing  \textbf{11}(2),  138--148 (1999).
  \doi{10.1287/ijoc.11.2.138},
  \url{https://pubsonline.informs.org/doi/pdf/10.1287/ijoc.11.2.138}

\bibitem{cook_combinatorial_1998}
Cook, W., Cunningham, W., Pulleyblank, W., Schrijver, A.: Combinatorial
  optimization. Wiley-{Interscience} series in discrete mathematics and
  optimization, Wiley, New York (1998)

\bibitem{gunluk_exact_2011}
Cook, W., Koch, T., Steffy, D.E., Wolter, K.: An exact rational mixed-integer
  programming solver. In: Günlük, O., Woeginger, G.J. (eds.) Integer
  programming and combinatoral optimization, vol.~6655, pp. 104--116. Springer
  Berlin Heidelberg, Berlin, Heidelberg (2011),
  \url{http://link.springer.com/10.1007/978-3-642-20807-2_9}

\bibitem{edmonds1965b}
Edmonds, J.: Maximum matching and a polyhedron with $0,1$ vertices. J. of Res.
  the Nat. Bureau of Standards  \textbf{69~B},  125--130 (1965)

\bibitem{edmonds1965a}
Edmonds, J.: Paths, trees, and flowers. Canad. J. Math.  \textbf{17},  449--467
  (1965)

\bibitem{GleixnerAmbrosM.2016IRfL}
Gleixner, A.M., Steffy, D.E., Wolter, K.: Iterative refinement for linear
  programming. INFORMS Journal on Computing  \textbf{28}(3),  449--464 (2016).
  \doi{10.1287/ijoc.2016.0692}

\bibitem{GLS}
Grötschel, M., Lovász, L., Schrijver, A.: Geometric algorithms and
  combinatorial optimization. Springer (1988),
  \url{http://eudml.org/doc/204222}

\bibitem{kielstra_code_2019}
Kielstra, P.M.: Code for revisiting a cutting plane method for perfect
  matchings (Aug 2019). \doi{10.5281/zenodo.3374971}

\bibitem{kolmogorov_blossom_2009}
Kolmogorov, V.: Blossom {V}: A new implementation of a minimum cost perfect
  matching algorithm. Mathematical Programming Computation  \textbf{1}(1),
  43--67 (Jul 2009). \doi{10.1007/s12532-009-0002-8},
  \url{http://link.springer.com/10.1007/s12532-009-0002-8}

\bibitem{padberg_rao_1982}
Padberg, M.W., Rao, M.R.: Odd minimum cut-sets and b-matchings. Math. Oper.
  Res.  \textbf{7}(1),  67--80 (Feb 1982). \doi{10.1287/moor.7.1.67},
  \url{http://dx.doi.org/10.1287/moor.7.1.67}

\bibitem{schrijver_theory_2000}
Schrijver, A.: Theory of Linear and Integer Programming. Wiley-{Interscience}
  series in discrete mathematics and optimization, Wiley, Chichester, reprinted
  edn. (2000), oCLC: 247967491

\end{thebibliography}

\end{document}